\documentclass[11pt,letterpaper]{amsart}

\usepackage{amsmath,amssymb,amsthm,mathtools,mathrsfs}
\usepackage{bm}
\allowdisplaybreaks
\usepackage[all]{xy}
\usepackage{hyperref}

\usepackage{here}
\usepackage{time}
\usepackage{xcolor}
\usepackage[capitalize,nameinlink,noabbrev,nosort]{cleveref}
\hypersetup{
	colorlinks=true,       
	linkcolor=blue,          
	citecolor=blue,        
	filecolor=blue,      
	urlcolor=blue,           
}

\makeatletter
\@namedef{subjclassname@2020}{\textup{2020} Mathematics Subject Classification}
\makeatother

\newtheorem{theoremcounter}{Theorem Counter}[section]

\theoremstyle{remark}
\newtheorem{remark}{Remark}

\theoremstyle{definition}
\newtheorem{definition}[theoremcounter]{Definition}

\newtheorem{lemma}[theoremcounter]{Lemma}
\newtheorem{fact}[theoremcounter]{Fact}
\newtheorem{proposition}[theoremcounter]{Proposition}
\newtheorem{corollary}[theoremcounter]{Corollary}

\newtheorem{theorem}[theoremcounter]{Theorem}

  \newtheoremstyle{mystyle}
    {}
    {}
    {\normalfont}
    {}
    {\bfseries}
    {}
    { }
    {}
\theoremstyle{mystyle}
\newtheorem*{proposition*}{Proposition}
\newtheorem*{corollary*}{Corollary}
\newtheorem*{conjecture*}{Conjecture}
\newtheorem*{Ac}{Acknowledgement.}
\newtheorem*{AI}{Use of AI.}
\newtheorem*{theorem*}{Theorem}

\numberwithin{equation}{section}

\newcommand{\WD}{\mathrm{WD}}
\newcommand{\JL}{\mathrm{JL}}
\newcommand{\sym}{\mathrm{sym}}
\newcommand{\Id}{\mathrm{Id}}
\newcommand{\CM}{\mathcal{M}}
\newcommand{\FL}{\mathrm{FL}}
\newcommand{\e}{\varepsilon}
\newcommand{\cadist}{\mathrm{Irr}_{\mathrm{cusp,A-dist}}}
\newcommand{\gadist}{\mathrm{Irr}_{\mathrm{gen,A-dist}}}

\newcommand{\sdg}{\mathrm{Irr}_{\mathrm{gen,sd}}}
\newcommand{\ord}{\mathrm{ord}}
\newcommand{\A}{\mathbb{A}}
\newcommand{\ad}{\mathrm{Ad}}
\newcommand{\gen}{\mathrm{Irr}_{\mathrm{gen}}}

\newcommand{\symp}{\mathrm{Irr}_{\mathrm{gen,symp}}}
\newcommand{\sympc}{\mathrm{Irr}_{\mathrm{cusp,symp}}}

\newcommand{\Q}{\mathbb{Q}}
\newcommand{\R}{\mathbb{R}}
\newcommand{\C}{\mathbb{C}}

\DeclareMathOperator{\Sp}{Sp}

\DeclareMathOperator{\GL}{GL}

\def\1{\mathbf{1}}
\def\cusp{\mathrm{Irr}_{\mathrm{cusp}}}

\DeclareMathOperator{\diag}{diag} 
\DeclareMathOperator{\tr}{tr}
\DeclareMathOperator{\Hom}{Hom}

\DeclareMathOperator{\re}{Re}
\DeclareMathOperator{\Irr}{Irr}

\begin{document}

\title[Cuspidal distinction]{Cuspidal distinction via intertwining periods}

\author[]{Nadir Matringe}
\address{Nadir Matringe. Shanghai Institute for Mathematics
and Interdisciplinary Sciences, Block A, International Innovation Plaza, No. 657 Songhu Road, Yangpu District, Shanghai, China}
\email{nadirmatringe@outlook.fr}



\begin{abstract}

Local intertwining periods have so far been used to reduce distinction of a general irreducible representation to distinction of cuspidal ones, for which other tools — type theory, local and global trace formulae — take over. The main contribution of this paper is to prove that they can be used to decide distinction of cuspidal representations themselves. We develop this approach with three well-studied examples for which the ambient group is an inner form of $\GL_n$ over a $p$-adic field, namely linear models, their non-split twisted version, and the quaternionic symplectic model. As consequences, we obtain proofs valid in arbitrary residual characteristic of the cuspidal case of the  epsilon dichotomy conjecture of Prasad and Takloo-Bighash and of Prasad's conjecture on quaternionic symplectic models. The proof in the quaternionic symplectic model case, although going through intertwining periods, uses the simplest ones, namely closed intertwining periods, and is disappointingly short and elementary. 
\end{abstract}

\maketitle
\renewcommand{\thefootnote}{}
\footnote[0]{\today\ \now}

\renewcommand{\thefootnote}{\arabic{footnote}}
 \setlength{\leftmargini}{20pt}

\section{Introduction}

Analytic properties of local intertwining periods have proven to be very useful in reducing distinction properties of irreducible representations to distinction of cuspidal representations, see for example \cite{FLO}, \cite{MatJFA}, \cite{SX24} and \cite{MSSS}. In all of this work they play the same role: distinction is propagated between an induced representation and its inducing data, so the argument terminates at the cuspidal level, which must then be handled by other means such as type theory or local and global trace formulae. To the best of our knowledge, local intertwining periods had never been used locally to decide the cuspidal case itself, only to reduce to it. The main contribution of this paper is to prove that they can.

Our method rests on one general principle. Given a cuspidal representation $\pi$, we consider induced representations $I(s)$ whose cuspidal support is built out of $\pi$, together with the intertwining periods they carry. The standard intertwining operator relates $I(s)$ and $I(-s)$, and comparing a period on one side with its transport to the other side, either through a functional equation or through an isomorphism at a point of irreducibility, turns out to encode the distinction of $\pi$ itself.

We test our method on three meaningful examples where the ambient group is an inner form of 
$\GL_n$ over a $p$-adic field: linear models and their non-split twisted version, and the quaternionic symplectic model. 

We first describe our result for linear models, which is the most substantial. Let $F$ be a finite extension of $\Q_p$, and $D$ be a central division $F$-algebra of dimension $d^2$. Let $m$ be a positive integer, set $n=md$, and write $G=G_m=\GL_m(D)$. Let $A$ be a quadratic \'etale $F$-algebra, i.e. either a quadratic extension $E$ of $F$, or $A\simeq F\times F$, and let $\eta_{A/F}$ be the character of $F^\times$ with kernel $N_{A/F}(A^\times)$. 
Assume that $A$ embeds in $\CM_m(D)$, and when $A\simeq F\times F$, assume moreover that $m$ is even and that the embedding is conjugate to $(x,y)\mapsto \diag(xI_{m/2},yI_{m/2})$. 
Let $H_A$ denote the centralizer of $A^\times$ in $G$.
We say that a smooth representation $\pi$ of $G$ is \emph{$A$-distinguished} if \[\Hom_{H_A}(\pi,  \C)\neq 0.\]

We refer to \cite[Definition 2.1]{ALMSY} for the definition of a generic representation of $G$. If $\pi$ is generic, then it has a well-defined Jacquet-Langlands transfer $\JL(\pi)$, which is a generic representation of $\GL_n(F)$, and we define the L-parameter $\phi_\pi$ of $\pi$ by setting \[\phi_\pi:=\phi_{\JL(\pi)}:\WD_F\to \GL_n(\C).\] Let $\psi:F\to \C^\times$ be a non-trivial character. If $\phi:\WD_F\to \Sp_n(\C)$ is a symplectic L-parameter, we set \[\e_A(\phi):=\eta_{A/F}(-1)^{n/2}\epsilon(1/2,\phi,\psi)\epsilon(1/2,\eta_{A/F}\otimes\phi,\psi),\] where the $\epsilon$ factor above is the usual Deligne-Langlands epsilon factor as defined in \cite{Tate}. This quantity does not depend on the choice of the character $\psi$ since $\phi$ is symplectic. In particular $\e_A(\phi)=1$ when $A\simeq F\times F$. 

Note that if $\pi$ is a cuspidal representation of $G$, then it is generic. Our first main result establishes with a new method the converse implication of the following statement without any restriction on the residual characteristic.

\begin{theorem}\label{cusp thm} 
	Let $\pi$ be a cuspidal representation of $G_m$.  The representation $\pi$ is $A$-distinguished if and only if $\phi_\pi$ is symplectic and $\varepsilon_A(\phi_\pi)=(-1)^m$. 
\end{theorem}

\begin{remark}
	When $A\simeq F\times F$, the root number $\varepsilon_A(\phi)$ is identically equal to $1$, and the condition $\varepsilon_A(\phi)=(-1)^m$ just means that $m$ is even.  It is thus this constant value that forces the parity of $m$: the root number still decides, but always in  the same way. 	
\end{remark}

 When $A\simeq F\times F$, Theorem \ref{cusp thm} follows from \cite[Theorem 1.3]{BPW} and the local unfolding principle of \cite{SV} (see \cite[Theorem 5.1]{MatJNT} and \cite[Remark 3.10]{ALMSY}). 
When $A=E$, the direct implication of Theorem \ref{cusp thm} is proved in Xue \cite{Xue21}. Its more intricate converse is proved in S\'echerre \cite{Sec24} under the assumption $p\neq 2$, and in \cite{Xue21} under the assumption that $d\leq 2$. Following earlier work of Feigon, Martin and Whitehouse in \cite{FMW}, Xue proves the direct implication by proving enough of the global Guo--Jacquet conjecture (\cite{Guo96}), using a simple form of the global trace formula proposed by Guo in \cite{Guo96}. We observe in passing that the direct implication of the Guo--Jacquet conjecture is now a consequence of the recent \cite[Theorem 1.4]{MOY}, established by a new method which can be thought of as a global analogue of the approach developed in this paper. As for the converse, S\'echerre's proof relies on type theory, which leads to the restriction $p\neq 2$. Theorem \ref{cusp thm} when $A=E$ should also follow from a local trace formula approach proving a special case of the conjecture in \cite{Wan}, as explained and started by Suzuki in \cite{Strace}.

Actually, the more general Theorem \ref{conj:PTB} below holds for generic representations. We record it both as a consequence of Theorem \ref{cusp thm} and for later reference, as it is for example used in \cite{MOY} to prove a partial converse to the Guo--Jacquet conjecture. When $A=E$, the reduction of the statement of Theorem \ref{conj:PTB} to Theorem \ref{cusp thm} is done in  \cite{Sjnt} and \cite{SX24}, using Mackey theory and intertwining periods, with an input from \cite{Ch}. When $A\simeq F\times F$, the statement is proved in \cite{ALMSY} (see Section \ref{sec:generic epsilon direct} below for the details), where it is reduced to the discrete series representations statement using the same tools, and we observe that it could further be reduced to cuspidal representations using intertwining periods again, following closely \cite{SX24} or \cite{MatJFA}.

\begin{theorem}\label{conj:PTB}
Let $\pi$ be a generic representation of $G$, and $\phi:=\phi_\pi$.
\begin{itemize}
\item[(1)] If $\pi$ is $A$-distinguished, then $\phi$ takes values in $\Sp_n(\C)$.
\item[(2)] Suppose $\phi$ takes values in $\Sp_n(\C)$ and we write it in the following form:
    \[
    \phi = \bigoplus_{i\in I_\phi} \phi_i \oplus \bigoplus_{j\in J_\phi}(\phi_j\oplus\phi_j^\vee),
    \]
where $I_\phi,  J_\phi$ are finite sets and $\phi_k\,\colon WD_F\to\GL_{n_k}(\C)$ is a discrete $L$-parameter such that
\begin{itemize}
\item for each $i\in I_\phi$,  $n_i$ is even and $\phi_i$ takes values in $\Sp_{n_i}(\C)$,
\item for any $i,  i'\in I_\phi$,  we have $\phi_i\not\simeq\phi_{i'}$ if $i\neq i'$. 
\end{itemize}
Then $\pi$ is $A$-distinguished if and only if we have 
    \[
    \e_A(\phi_i) = (-1)^{n_i/d},   \qquad  i\in I_\phi.
    \]
\end{itemize}
\end{theorem}

When $A=E$, the above statement was predicted by Prasad and Takloo-Bighash in \cite[Conjecture 1]{PTB}, in a slightly more general form involving a character twist. It generalizes the Saito--Tunnell criterion of \cite{Sai} and \cite{Tu} for local toric periods on $\GL_2(F)$. 

\begin{remark}
The statement of Theorem \ref{conj:PTB} also makes sense over archimedean local fields. Its more general version involving a character twist was settled by Tamori and Suzuki \cite{ST23} when $A=E$. When $A\simeq F\times F$, the untwisted statement follows from the necessary condition proved by the same authors in \cite[Appendix D]{ALMSY} and the converse implication in \cite[Theorem 5.4]{MOYjfa}.  
\end{remark}

We now summarize our proof of the converse implication of Theorem \ref{cusp thm}. Our main discovery is that the  functional equation of local intertwining periods encodes most of the information about the distinction of cuspidal representations. Namely, let $\pi$ be a cuspidal representation of $G$. It is well-known to experts that $\pi$ is $A$-distinguished if and only if the open intertwining period \[J_{\pi,A}(s):\pi[s/2]\times \pi^\vee[-s/2]\to \C\]  has a pole at $ s=0$  (see Section \ref{sec:signFE} for this fact, and Sections \ref{sec:NIO}--\ref{sec:OIP} for unexplained notation). 
Now assume that $\pi$ has a symplectic Langlands parameter, so that in particular $\pi=\pi^\vee$. Then by multiplicity at most one, we have a functional equation 
\[J_{\pi,A}(-s)\circ M_{\pi}(s)=\alpha_{\pi,A}(s) J_{\pi,A}(s),\] where 
\[M_{\pi}(s):\pi[s/2]\times \pi[-s/2]\to \pi[-s/2]\times \pi[s/2]\] is the standard intertwining operator. A first main observation, proved in Theorem \ref{thm:key1}, is that if we normalize $\alpha_{\pi,A}(s)$ by the Shahidi exterior and symmetric square gamma factors as \[\beta_{\pi,A,\psi}(s)=\gamma(-s,\JL(\pi),\wedge^2,\psi)\gamma(s,\JL(\pi),\sym^2,\psi)\alpha_{\pi,A}(s),\] then $J_{\pi,A}(s)$ has a pole at $s=0$ if and only if $\beta_{\pi,A,\psi}(0)=1$, in other words:
\begin{equation}\label{eq:iff} \pi \mbox{ is distinguished} \iff \beta_{\pi,A,\psi}(0)=1 .\end{equation} By Theorem \ref{thm:almost explicit constant}, we have
\[
 \beta_{\pi,A,\psi}(s)
 =c_A(\pi,\psi)\gamma((s+1)/2,\JL(\pi)_A,\psi_A).
\]
It therefore remains to determine the constant $c_A(\pi,\psi)$, which we do by a local to global argument in Section \ref{sec:conv}. In order to achieve this, in Lemma \ref{lm:globlz}, we globalize $A/F$ as the completion at a place $v_0$ of a quadratic extension $\mathfrak{E}/\mathfrak{F}$ of number fields, and with the help of new results of Takanashi and Wakatsuki \cite{TW}, we globalize $\pi$ as the local component $\pi=\Pi_{v_0}$ of a cuspidal automorphic representation $\Pi$ with symplectic Jacquet-Langlands transfer, in such a way that $\Pi_v$ is moreover $\mathfrak{E}_v$-distinguished for all places $v\neq v_0$ of $\mathfrak{F}$. Then, in order to put the globalization argument to good use,
we prove in Proposition \ref{prop:sign dist} that the direct implication in \eqref{eq:iff} 
holds under the hypotheses of Theorem \ref{thm:pole order lip}. From this we deduce a formula for $c_{\mathfrak{E}_v}(\Pi_v,\Psi_v)$ for all places $v\neq v_0$ of the number field globalizing $F$ in Corollary \ref{cor:key}. The combination of Proposition \ref{prop:sign dist} and Corollary \ref{cor:key}, i.e. the formula for $c_{\mathfrak{E}_v}(\Pi_v,\Psi_v)$, is our second main observation, and we appeal to the main local result of \cite{MOY} in its proof, as well as to the direct implication of Theorem \ref{conj:PTB} (which is already proven as we recall in Section \ref{sec:generic epsilon direct}). Finally, the product of all $c_{\mathfrak{E}_v}(\Pi_v,\Psi_v)$, including $c_A(\pi,\psi)$, being equal to one, as we prove in Lemma \ref{lm:product of c is one} thanks to another result of \cite{MOY}, we deduce in Theorem \ref{thm:key} the expected formula for $c_A(\pi,\psi)$. The final formula that we obtain is that if $\pi$ is cuspidal with symplectic L-parameter:
\[\beta_{\pi,A,\psi}(0)=(-1)^m\e_A(\phi_\pi).\]  Theorem \ref{cusp thm} follows; it is Corollary \ref{cor:main}.

\begin{remark}
Although we do not do this, we mention that we could include the Galois model, and obtain by the exact same method as that for linear models, the result of \cite[Theorem 0.4]{FH}, later vastly generalized in \cite[Theorem 1]{BPgal}. 
\end{remark}

We finally turn to quaternionic symplectic models. We specialize $D$ to the
quaternion division algebra over $F$, but on the other hand we allow $F$ to be any non-archimedean local field of charactristic not two. 

\begin{theorem}\label{thm:intro quaternionic converse}
Let $\pi$ be a cuspidal irreducible representation of $G$. If $\JL(\pi)$ is not cuspidal, then $\pi$ is
$\Sp_m(D)$-distinguished.
\end{theorem}

This theorem is proved under the assumption that $F$ has characteristic zero and $p$ is odd in \cite{MSSS}, using type theory and a global argument due to Verma (\cite{Verma}). Thanks to \cite{Verma} and \cite{MSSS}, this completes the proof of Prasad's conjecture on quaternionic symplectic models in characteristic zero.

\begin{corollary}\label{cor discrete}
Let $F$ be a non-archimedean local field of characteristic zero, and let $\pi$ be a discrete series representation of $G$. Then $\pi$ is $\Sp_m(D)$-distinguished if and only if it is cuspidal but $\JL(\pi)$ is not cuspidal.  
\end{corollary} 

The proof of Theorem \ref{thm:intro quaternionic converse} is
entirely local and carried out in Section
\ref{sec:quaternionic symplectic}. It is in fact very simple-minded. The hypothesis that $\JL(\pi)$ is
not cuspidal is equivalent to $\JL(\pi)=\operatorname{St}_2(\rho)$ for a
cuspidal representation $\rho$ of $\GL_m(F)$, and also to the equality
$l_\pi=2$, where $l_\pi$ denotes the reducibility number of the cuspidal
line of $\pi$. We again consider the induced representation 
$
I(s):=\pi[s]\times\pi[-s]$, but this time at the points $s=\pm 1/2$. Since $l_\pi=2$, both $I(\pm 1/2)$ are irreducible, hence isomorphic, with \[M:I(-1/2)\simeq I(1/2).\] Moreover $I(1/2)$ is equipped with a closed intertwining period \[L\in \Hom_{\Sp_{2m}(D)}(I(1/2),\C)-\{0\}\] given by compact integration. In particular \[L\circ M\in  \Hom_{\Sp_{2m}(D)}(I(-1/2),\C)-\{0\},\] but the geometric lemma à la Offen, in the explicit form given by Sharma and Verma in \cite{ShV}, shows at once that $I(-1/2)$ is distinguished only if (and actually if and only if, by the theory of open intertwining periods) the inducing representation $\pi$ is distinguished. This argument is in fact already implicit in \cite[Section~4]{Verma} and \cite[Section~5]{ShV25}, where it is used in the opposite direction and for different purposes.

\begin{Ac}
We thank Miyu Suzuki for sharing her insights and notes on the positive characteristic case of the epsilon dichotomy conjecture, which served as an impetus for completing the proof in characteristic zero. The idea of the proof of this part occurred to the author while he was visiting Huajie Li at the YMSC of Tsinghua University. We thank him and the YMSC for the hospitality. 
\end{Ac}

\begin{AI}
ChatGPT 5.6, ChatGPT 6 and Claude Opus 5.5 were used in the preparation of Section~\ref{sec:quaternionic symplectic}.  Starting from the author's idea of using a functional equation from \cite{MSSS}, and to seek an adaptation of a result from \cite{MatJFA}, they assisted him in working out a residue and cokernel argument in the previous proof. However, once the appeal to a result of \cite{Sha26}, which served to exclude the distinction of a certain induced representation, was replaced by the geometric lemma for cuspidally induced representations, the very short and elementary argument of the present proof of Theorem~\ref{thm:quaternionic converse} became evident. Finally, ChatGPT 6, asked to review the whole paper, spotted an issue in an earlier version of Theorem~\ref{thm:pole order lip} when $E/F$ is a quadratic extension. It suggested an interesting fix, valid for a large class of representations, which we did not include: we simply restricted the class of representations to one to which the original statement applies, this class being sufficient for our globalization argument.
\end{AI}

\section{Notation}\label{sec:not}

Let $F$ be a local field of characteristic zero, possibly archimedean. We fix a nontrivial character $\psi:F\to \C^\times$.  Let $A$ be an $F$-separable algebra of dimension two, i.e. either $A/F$ is a quadratic extension, or $A\simeq F\times F$. We set $\psi_A:=\psi\circ \tr_{A/F}$, and denote by $\eta_{A/F}$ the character of $F^\times$ with kernel the norms of $A^\times$, in particular $\eta_{A/F}$ has order two when $A/F$ is quadratic, whereas it is trivial when $A\simeq F\times F$. We note that $A$ can be written $F[u]$, with $\tr_{A/F}(u)=0$, i.e. $u^2\in F^\times$. Such an element $u$ is unique up to scaling by an element in $F^\times$. Let $D$ be a central division algebra over $F$, such that $[D:F]=d^2$ with $d\geq 1$. We set $\CM:=\CM_m(D)$ for $m\geq 1$, and $n:=md$. We put $\psi_{\CM}:=\psi\circ \tr_{\CM/F}$. We assume that $A$ embeds in $\CM$ as an $F$-subalgebra, in such a way that $\tr_{\CM/F}(u)=0$, and we realize the embedding explicitly, as explained in \cite[Section 2.5]{MOY}. In particular $n$ is even, and if $A/F$ is quadratic, this parity condition is also sufficient for the assumption $\tr_{\CM/F}(u)=0$ to hold. If $A\simeq F\times F$, the assumption holds if and only if $m$ is even and $u$ is conjugate to $\lambda\diag(I_{m/2},-I_{m/2})$ with $\lambda\in F^\times$. We set $G_m=G=\CM^\times$, and $H_{A,m}=H_A=Z_{\CM}(A)^\times$. The group $H_A$ is the subgroup of $G$ fixed by the inner involution $\theta:=\ad(u)$ of $G$. We denote by $\nu:G\to \C^\times$ the restriction of the composition of the normalized absolute value $|\ |$ on $F$ with the reduced norm on $\CM$. We denote by $P_{m,m}$ the upper block triangular parabolic subgroup of $G_{2m}$ corresponding to the partition $(m,m)$, and by $M_{m,m}$ its block diagonal Levi subgroup. We denote by $x_0$ the representative of the unique open double coset $P_{m,m}x_0H_{A,2m}$ fixed in \cite[Section 4.1]{MOY}. In particular the involution $\theta_0(g):=x_0\theta(x_0^{-1}gx_0)x_0^{-1}$ of $G_{2m}$ is such that $\theta_0(P_{m,m})$ is opposite to $P_{m,m}$ with respect to $M_{m,m}$. We denote by $\FL(G)$ the objects of the category of smooth admissible representations of $G$ of finite length, where for archimedean $F$, admissible means admissible, Fr\'echet, and of moderate growth (see \cite[Chapter 11]{RRG2}). For $\pi\in \FL(G)$, we write $\pi^\vee$ for its smooth contragredient, and we define  $\ell\in \Hom(\pi\otimes \pi^\vee,\C)$ by 
\[\ell(v\otimes v^\vee)=\langle v, v^\vee \rangle\] for $v \in \pi,\ v^\vee \in \pi^\vee$. We use the Bernstein-Zelevinsky product notation $\times$ for normalized parabolic induction. 

\section{Generic and distinguished representations}\label{sec:gendist}

An element $\pi\in \Irr(G)$ is called generic if it satisfies the conditions of \cite[Definition 2.1]{ALMSY}. We denote by $\gen(G)$ the set of isomorphism classes of generic representations of $G$. In particular, for $\pi\in \gen(G)$, its Jacquet-Langlands transfer $\JL(\pi)\in \gen(\GL_n(F))$ is well-defined. Let $\WD_F$ be the Weil-Deligne group of $F$. By definition, the L-parameter 
\[\phi_\pi:\WD_F\to \GL_n(\C)\] of $\pi$ is that of its Jacquet-Langlands transfer $\JL(\pi)$ \cite{HT,Hen,DKV}, and we say that $\pi$ is symplectic if $\phi_\pi(\WD_F)\subseteq \Sp_n(\C)$. We denote by $\symp(G)$, resp. $\sdg(G)$, the subset of $\gen(G)$ whose elements are symplectic, resp.\ self-dual; $\pi\in \sdg(G)$ if $\pi^\vee= \pi$. In particular $\symp(G)\subseteq \sdg(G)$. When $F$ is $p$-adic and only in this case, we set $\sympc(G)$ to be the cuspidal part of $\symp(G)$. In many statements below, if we mention cuspidal representations, it will be implicit that for this part of the statement, the field $F$ has to be $p$-adic. 

We say that $\pi\in \gen(G)$ is \textit{$H_A$-distinguished}, or simply \textit{$A$-distinguished} if $\Hom_{H_A}(\pi,\C)\neq \{0\}$. We denote by $\gadist(G)$ the subset of $\gen(G)$ whose elements are $A$-distinguished. 
As follows from \cite[Theorem 1.1]{ALMSY}, which relies on many previous results, we have
 \[\gadist(G)\subseteq \symp(G),\] and moreover $\Hom_{H_A}(\pi,\C)\simeq \C$ whenever $\pi\in \gadist(G)$. The above inclusion is actually an equality when $A\simeq F\times F$ and $G=\GL_n(F)$ according to \cite[Corollary 3.15]{MatCrelle}, \cite[Appendix D]{ALMSY} and \cite[Theorem 5.4]{MOYjfa}. We set $\cadist(G)$ to be the cuspidal part of $\gadist(G)$.

\section{Normalized intertwining operators}\label{sec:NIO}

 For $s$ a complex number and $\pi\in \FL(G)$, we set $\pi[s]:=\nu^{s}\otimes \pi$. If $\pi$ has a central character, we denote it by $\omega_\pi$. There is a standard meromorphic intertwining operator 
\[M_\pi(s):\pi[s/2]\times \pi[-s/2]\to \pi[-s/2]\times \pi[s/2],\] the basic properties of which are for example recalled in \cite[2.2.2]{MOY}.
The definition of such an operator necessitates the choice of a Haar measure on $D$, and we take the $\psi_D$-self-dual Haar measure, where $\psi_D:=\psi\circ\tr_{D/F}$. When $\pi\in \gen(G)$, we have at our disposal the exterior and symmetric square local constants attached to $\JL(\pi)$, as well as the local constants attached to the pair $(\JL(\pi),\JL(\pi))$. They can be  defined either by the Langlands-Shahidi method or via the Artin local constants of their Langlands parameters. They are all known to agree, in particular thanks to \cite{Hext} and \cite{CST}. We have the equalities 

\begin{equation}\label{eq:L} L(s,\JL(\pi),\JL(\pi))=L(s,\JL(\pi),\wedge^2)L(s,\JL(\pi),\sym^2)\end{equation}

and

\begin{equation}\label{eq:gamma}\gamma(s,\JL(\pi),\JL(\pi),\psi)=\gamma(s,\JL(\pi),\wedge^2,\psi)\gamma(s,\JL(\pi),\sym^2,\psi).\end{equation}

We set \[k_{\pi}:=\ord(s=0,\ L(s,\JL(\pi),\wedge^2))\] to be the order of the pole of $L(s,\JL(\pi),\wedge^2)$ at $s=0$. Throughout Sections \ref{sec:NIO}--\ref{sec:conv}, $\ord$ denotes a \emph{pole} order, so that $\ord=0$ means regular and $\ord=1$ means a simple pole; in Section \ref{sec:quaternionic symplectic} we avoid this notation altogether and describe analytic behaviour in words. We warn the reader that the normalization of $M_\pi(s)$ below is not the usual Langlands-Shahidi normalization.
 
\begin{definition}
For $\pi\in \gen(G)$, we set \[N_{\pi,\psi}(s):=\gamma(-s,\JL(\pi),\wedge^2,\psi)\gamma(s,\JL(\pi),\sym^2,\psi)M_\pi(s).\]
\end{definition}

The property that we need from this normalization is the following result. 

\begin{proposition}\label{prop Shahidi}
Assume that either $\pi\in \cusp(G)$ or that $G=\GL_n(F)$ is split and that $\pi\in \gen(G)$. In both cases assume moreover that $\omega_{\pi}$ is trivial and that $\pi=\pi^\vee$. Then 
\[N_{\pi,\psi}(0)=(-1)^{k_\pi} \Id_{\pi\times \pi}.\]
\end{proposition}
\begin{proof}
First we assume that $G$ is split and that $\pi$ is generic. In this case our claim follows from the results of Shahidi. Let $W_{\pi,\psi}(s)$ be the $\psi$-Whittaker functional defined on 
$\pi[s/2]\times \pi[-s/2]$ in \cite[Proposition 3.1]{Shproof} and \cite[p. 987, Corollary]{Shreal}. It is entire with respect to $s$, and nonzero for any given $s\in \C$. Moreover, since $\omega_{\pi}$ is trivial, it follows from \cite[Theorem 3.5]{Shproof} that 
\[W_{\pi,\psi}(-s)\circ M_\pi(s)=\gamma(s,\pi,\pi,\psi)^{-1}W_{\pi,\psi}(s),\] i.e. 
\begin{equation}\label{eq:sh} W_{\pi,\psi}(-s)\circ N_{\pi,\psi}(s)=\frac{\gamma(-s,\pi,\wedge^2,\psi)}{ \gamma(s,\pi,\wedge^2,\psi)}W_{\pi,\psi}(s)\end{equation} according to Equation \eqref{eq:gamma}. But by assumption $\pi$ is generic, so the pair L-function $L(s,\pi,\pi)=L(s,\pi,\pi^\vee)$ is holomorphic at $s=1$, and $L(s,\pi,\wedge^2)$ as well according to Equation \eqref{eq:L}. Since  \[\gamma(s,\pi,\wedge^2,\psi)=\epsilon(s,\pi,\wedge^2,\psi)\frac{L(1-s,\pi,\wedge^2)}{L(s,\pi,\wedge^2)},\] and because $\pi\times \pi$ is irreducible, the result follows from evaluating Equation \eqref{eq:sh} at $s=0$.

It remains to treat the case where $\pi\in \cusp(G)$ and $G$ is not necessarily split. In this case it follows from \cite[Lemma 2.1]{AC} that the operator 
\[\frac{\epsilon(s,\JL(\pi),\JL(\pi),\psi)L(s+1,\JL(\pi),\JL(\pi))}{L(s,\JL(\pi),\JL(\pi))}M_\pi(s)\] is holomorphic and unitary at $s=0$, whereas it follows from \cite{Olsh}, as recalled in \cite[Theorem A.2]{LM}, that it is a positive scalar operator, hence the identity. Equivalently, the factor $\gamma(s,\JL(\pi),\JL(\pi),\psi)M_\pi(s)$ gives the identity of $\pi\times \pi$  when evaluated at $s=0$, and the result now follows as in the split generic case, since $\JL(\pi)$ is generic. 
\end{proof}

\section{Open intertwining periods}\label{sec:OIP}

There is a relative version of the above intertwining operator, attached to symmetric pairs, which we call open intertwining periods here. Let $\pi\in \FL(G)$, and $f_s\in \pi[s/2]\times \pi^\vee[-s/2]$ be a holomorphic section. It follows from \cite{BD}, \cite{BrD}, and \cite{CrD} that the integral  
\[J_{\pi,A}(s,f_s):=\int_{x_0^{-1}P_{m,m}x_0\cap H_{A,2m}\backslash H_{A,2m}}\ell(f_s(x_0 h))d\mu(h)\] converges absolutely for $\re(s)>R_{\pi}$, where $R_\pi\in\R$ depends only on $\pi$, and defines a meromorphic family, in a suitable sense, of $H_{A,2m}$-invariant linear forms on $\pi[s/2]\times \pi^\vee[-s/2]$. We refer to \cite[Section 5]{MOY} for more details. In the present work, the choice of the invariant measure $\mu$ does not matter, since such a measure shows up on both sides of the functional equation \eqref{eq:lipfe} below. We use the notation  \[r_\pi:=\ord(s=0,J_{\pi,A}(s)),\] where $\ord(s=0,J_{\pi,A}(s))$ makes sense by \cite[Section 5]{MOY}. In particular, it makes sense to define the regularized intertwining period 
\[J_{\pi,A}^*=(s^{r_\pi}J_{\pi,A}(s))_{|s=0}\in \Hom_{H_{A,2m}}(\pi\times \pi^\vee,\C)-\{0\}.\]

Assume further that $\pi\in \sdg(G)$. By multiplicity one for the pair $(G_{2m},H_{A,2m})$, the intertwining period $J_{\pi,A}(s)$ satisfies a functional equation: 
\begin{equation}\label{eq:lipfe}
J_{\pi,A}(-s,M_{\pi}(s)f_s)=\alpha_{\pi,A}(s)J_{\pi,A}(s,f_s),
\end{equation} 
where $f_s$ is a holomorphic section of $\pi[s/2]\times \pi[-s/2]$, and $\alpha_{\pi,A}(s)$ is meromorphic in the variable $s$. Here we tacitly identified $\ell$ with a linear form on $\pi\otimes \pi$, via the choice of an isomorphism between $\pi$ and $\pi^\vee$. Nevertheless, we observe that since this identification has been made on both sides of the functional equation, the proportionality constant $\alpha_{\pi,A}(s)$ does not depend on it. 

We will crucially use two important results from \cite{MOY}. The first one is an expression of $\alpha_{\pi,A}(s)$ up to some nonzero constant, whereas the second determines the order of the pole order at $s=0$ of $J_{\pi,A}(s)$ when $\pi$ satisfies one of the assumptions of Theorem \ref{thm:pole order lip}. For $\pi\in \gen(G)$, we set  \[\gamma(s,\JL(\pi)_A,\psi_A)=\gamma(s,\phi_\pi,\psi)\gamma(s,\eta_{A/F}\otimes \phi_\pi,\psi),\] 
\[\epsilon(s,\JL(\pi)_A,\psi_A)=\epsilon(s,\phi_\pi,\psi)\epsilon(s,\eta_{A/F}\otimes \phi_\pi,\psi),\] 
and \[L(s,\JL(\pi)_A)=L(s,\phi_\pi)L(s,\eta_{A/F}\otimes \phi_\pi).\] 

\begin{theorem}\label{thm:almost explicit constant}
Let $\pi\in \sdg(G)$. Then there exists $c_A(\pi,\psi)\in \C^\times$ such that 
\[\alpha_{\pi,A}(s)=c_A(\pi,\psi)\frac{\gamma((s+1)/2,\JL(\pi)_A,\psi_A)}{\gamma(-s,\JL(\pi),\wedge^2,\psi)\gamma(s,\JL(\pi),\sym^2,\psi)}.\]
Moreover, if $G$ is $F$-split, $A/F$ is unramified, and $\pi$ and $\psi$ are  unramified too, then $c_A(\pi,\psi)=1$. 
\end{theorem}
\begin{proof}
This is a consequence of \cite[Remark 6.7 following the proof of Theorem 6.6 and Remark 6.9 following the proof of Corollary 6.8]{MOY} (\cite[Proof of Proposition 4.7]{SX24} for a special case). 
\end{proof}

\begin{theorem}\label{thm:pole order lip}
Let $\pi\in \gadist(G)$. Then
$r_\pi= k_\pi$ provided that we are in one of the following situations:
\begin{enumerate}
	\item $\pi$ is cuspidal,
	\item $G$ is split, $A\simeq F\times F$, and $\pi$ is unitary,
	\item $G$ is split, $A/F=E/F$ is quadratic, and $\pi$ is a unitary principal series induced from the triangular Borel subgroup of $G$. 
\end{enumerate}

\end{theorem}  
\begin{proof}
	Under these assumptions, $\pi$ belongs to $\Pi_D(-1/2,1/2)$:
	in the cuspidal case, its central character is trivial, so $\pi$
	is unitary. Moreover, in all three cases, $\pi$ is
	$H_A$-compatible in the sense of \cite[Definition~2.8]{MOY}.
	Since $L(s,\JL(\pi)_A)$ is holomorphic at $s=1/2$,
	\cite[Theorem~6.10]{MOY} gives $r_\pi=k_\pi$.
\end{proof}

\section{The sign in the functional equation of open intertwining periods}\label{sec:signFE}

In this section, for $\pi\in \gen(G)$, we study a sign occurring in the functional equation of $J_{\pi,A}(s)$, and relate it to analytic properties of $J_{\pi,A}(s)$ at $s=0$. 

\begin{definition}
Let $\pi\in \sdg(G)$. We set 
\[\beta_{\pi,A,\psi}(s):=\gamma(-s,\JL(\pi),\wedge^2,\psi)\gamma(s,\JL(\pi),\sym^2,\psi)\alpha_{\pi,A}(s)\]
\[=c_A(\pi,\psi)\gamma((s+1)/2,\JL(\pi)_A,\psi_A).\]
\end{definition}

We observe that this normalized proportionality constant is the one occurring in the functional equation of $J_{\pi,A}(s)$ with respect to normalized intertwining operators. Namely, for $\pi \in \sdg(G)$ and $f_s\in \pi[s/2]\times \pi[-s/2]$ a holomorphic section, Equation \eqref{eq:lipfe} reads:

\begin{equation}\label{eq:norm eq}
J_{\pi,A}(-s,N_{\pi,\psi}(s)f_s)=\beta_{\pi,A,\psi}(s)J_{\pi,A}(s,f_s)
\end{equation}

One key observation of this paper is the following.

\begin{proposition}\label{prop:sign dist}
Let $\pi\in \gadist(G)$ be as in Theorem \ref{thm:pole order lip}. Then $\beta_{\pi,A,\psi}(0)=1$.
\end{proposition}
\begin{proof}
In the cuspidal case, $\pi$ has trivial central character and is therefore
unitary as well. Multiplying Equation \eqref{eq:norm eq} by $s^{r_\pi}$
and letting $s$ tend to zero gives
\[
 (-1)^{r_\pi}J_{\pi,A}^*\circ N_{\pi,\psi}(0)
 =\beta_{\pi,A,\psi}(0)J_{\pi,A}^*.
\]
Since $J_{\pi,A}^*\neq 0$, Proposition \ref{prop Shahidi} and
Theorem \ref{thm:pole order lip} yield the triviality of $\beta_{\pi,A,\psi}(0)$. 
\end{proof}

In the proof below, we use the fact that the direct implication of Theorem \ref{conj:PTB} is known when $F$ is archimedean; indeed, in this case, the theorem is proved in \cite{ST23} when $A/F$ is quadratic, and the direct implication of interest to us is proved in \cite[Appendix D]{ALMSY} when $A\simeq F\times F$. 

\begin{corollary}\label{cor:key}
Let $\pi\in \gadist(G)$ be as in Theorem \ref{thm:pole order lip}, and let $c_A(\pi,\psi)$ be as in Theorem \ref{thm:almost explicit constant}. Then $c_A(\pi,\psi)=(-1)^m\eta_{A/F}(-1)^{n/2}$.
\end{corollary}
\begin{proof}
The unitarity assumption guarantees that $L(s,\JL(\pi)_A)$ is holomorphic at $s=1/2$ so that $\gamma(1/2,\JL(\pi)_A,\psi_A)=\epsilon(1/2,\JL(\pi)_A,\psi_A)$. The statement is thus a consequence of the direct implication in Theorem \ref{conj:PTB} and Proposition \ref{prop:sign dist}. 
\end{proof}

One part of the following result is given by \cite[Lemma 6.1 and Proposition 6.11]{MOY}, whereas the other direction follows from \cite[Lemma 5.1 and Proposition 4.13]{MOY}.

\begin{proposition}\label{prop:key}
Let $\pi\in\cusp(G)$. Then $J_{\pi,A}(s)$ has a pole at $s=0$, necessarily simple, if and only if $\pi$ is $A$-distinguished.
\end{proposition}

A second key observation of this paper is the following.  

\begin{theorem}\label{thm:key1}
Let $\pi\in\sympc(G)$. Then:
    \begin{enumerate}
    	\item if $J_{\pi,A}(s)$ has a pole at $s=0$, then $\beta_{\pi,A,\psi}(0)=1$,
    \item if $J_{\pi,A}(s)$ is holomorphic at $s=0$, then $\beta_{\pi,A,\psi}(0)=-1$. 
    \end{enumerate} 
\end{theorem}
\begin{proof}
First assume that $J_{\pi,A}(s)$ has a pole at $s=0$. Then $\pi$ is $A$-distinguished according to Proposition \ref{prop:key}, and $\beta_{\pi,A,\psi}(0)=1$ thanks to Proposition \ref{prop:sign dist}. For the other direction, if $J_{\pi,A}(s)$ is holomorphic at $s=0$ then $J_{\pi,A}(0)$ is automatically nonzero (see \cite{MOY}). Hence Equation \eqref{eq:norm eq} and  
Proposition \ref{prop Shahidi} imply that $\beta_{\pi,A,\psi}(0)=(-1)^{k_\pi}$. But $k_\pi\leq 1$ since $\JL(\pi)$ is square-integrable, and $k_\pi=1$ since $\pi$ is symplectic by assumption.
\end{proof}

\section{Proof of the epsilon dichotomy conjecture}

\subsection{The generic statement, direct implication}\label{sec:generic epsilon direct}

Here we recall from which results the direct implication of Theorem \ref{conj:PTB} follows.

\begin{proof}[Proof of Theorem \ref{conj:PTB}, direct implication]
When $A/F$ is quadratic it follows from \cite[Theorem 1.1]{Xue21} together with the classification in \cite{Sjnt}. When $A\simeq F\times F$, it follows from \cite[Theorem 4.6]{MOY} (which is proved in \cite{ALMSY} and builds in particular on \cite[Theorem 1.2]{BPW}), by observing that if $\pi$ is discrete and symplectic, then the root number condition amounts to saying that $m$ is even. 
\end{proof}

\subsection{The converse implication in the cuspidal case}\label{sec:conv}

Our proof of the converse implication is local to global, and uniform in $A$, which is why we include $A$ split as well. If $\mathfrak{F}$ is a number field, we denote by $\A_\mathfrak{F}$ its ring of adeles. If moreover $\mathfrak{D}$ is an $\mathfrak{F}$-central division algebra of index $d$, and $\Pi$ is a cuspidal automorphic representation of $\GL_m(\mathfrak{D}_{\A_\mathfrak{F}})$, we denote by $\JL(\Pi)$ the discrete automorphic representation of $\GL_n(\A_\mathfrak{F})=\GL_{md}(\A_\mathfrak{F})$ defined in \cite{BR}. We say that a cuspidal automorphic representation $\Pi'$ of $\GL_n(\A_\mathfrak{F})$ is symplectic if it admits a Shalika period as in \cite{JS90}. If so it follows for example from \cite{MatBLMS} and the argument in \cite[Proof of Theorem 10.6]{MOY} that $\Pi'_v$ is symplectic for any place $v$ of $\mathfrak{F}$. Our proof of the converse implication in the epsilon dichotomy conjecture is based on the determination of the constant $c_A(\pi,\psi)$ of Theorem \ref{thm:almost explicit constant}, when $\pi\in \sympc(G)$. One of the reasons why we can compute it by a local-global method is the following equality.

\begin{lemma}\label{lm:product of c is one}
Let $\mathfrak{E}/\mathfrak{F}$ be a quadratic extension of number fields, and $\Psi:\mathfrak{F}\backslash \A_\mathfrak{F}\to \C^\times$ be a nontrivial character. Let $\mathfrak{D}$ be an $\mathfrak{F}$-central division algebra. Assume moreover that 
$\mathfrak{E}$ is contained inside $\CM_m(\mathfrak{D})$. 
 If $\Pi$ is a self-dual cuspidal automorphic representation of $\GL_m(\mathfrak{D}_{\A_\mathfrak{F}})$ such that $\JL(\Pi)$ is cuspidal, then \[\prod_v c_{\mathfrak{E}_v}(\Pi_v,\Psi_v)=1,\] where the product is over all places of $\mathfrak{F}$. More specifically \[\prod_{v\in S} c_{\mathfrak{E}_v}(\Pi_v,\Psi_v)=1,\] where $S$ is the set of places of $\mathfrak{F}$ outside of which $\mathfrak{D}_v$ is split, and $\mathfrak{E}_v/\mathfrak{F}_v$, $\Pi_v$ and $\Psi_v$ are all unramified.
\end{lemma}
\begin{proof}
By \cite[Corollary 7.7]{MOY}, together with \cite[Corollary 3.10]{MOY} due to Shahidi (\cite{Shproof}), we deduce that 
\[\prod_{v\in S} \alpha_{\Pi_v,\mathfrak{E}_v}(s)=\prod_{v\in S}\frac{\gamma((s+1)/2, \JL(\Pi_v)_{\mathfrak{E}_v},(\Psi_v)_{\mathfrak{E}_v})}
{\gamma(-s,\JL(\Pi_v),\wedge^2,\Psi_v)\gamma(s,\JL(\Pi_v),\sym^2,\Psi_v)}.\] The result now follows from Theorem \ref{thm:almost explicit constant}.
\end{proof}

We recall that when $p\neq 2$, $\cadist(G)$ is non-empty according to \cite{Sec24}. When $A/F$ is unramified, and for any $p$, one easily deduces from \cite{Prfinite} that $\cadist(G)$ contains level zero elements, and we refer to \cite{CM} when $G=\GL_n(F)$. Actually $\cadist(G)$ is never empty according to a result of Beuzart-Plessis, and using it would slightly simplify the assumptions in Lemma \ref{lm:globlz} below, namely by removing the requirement that $w_0$ lies over an odd prime number in Assumption (2)(b). The following is our ad hoc globalization lemma. It relies on the new results of \cite{TW}, which avoid using endoscopic transfer from $\mathrm{SO}_{n+1}(\A_{\mathfrak{F}})$ to $\GL_n(\A_{\mathfrak{F}})$. One subtlety compared to other results of the same type that we know of is that we choose the quadratic extension only after globalizing the representation, so that its splitting behavior can be adapted to the resulting local components.

\begin{lemma}\label{lm:globlz}
Let $\pi\in \sympc(G)$. There exist:
\begin{enumerate}
\item a number field $\mathfrak{F}$ and a place $v_0$ such that $\mathfrak{F}_{v_0}\simeq F$,
\item an $\mathfrak{F}$-central division algebra $\mathfrak{D}$ such that:
\begin{enumerate}
\item $\mathfrak{D}_{v_0}/\mathfrak{F}_{v_0}\simeq D/F$, 
\item $\mathfrak{D}_{w_0}$ is a division algebra of Hasse invariant opposite to that of $D$, for some finite place $w_0\neq v_0$ of $\mathfrak{F}$ lying over an odd prime number,
\item $\mathfrak{D}_v$ is a split matrix algebra for any place $v$ of $\mathfrak{F}$ different from $v_0$ and $w_0$, 
\end{enumerate}
\item a cuspidal automorphic representation $\Pi$ of $\GL_m(\mathfrak{D}_{\A_\mathfrak{F}})$ such that:
\begin{enumerate}
\item $\JL(\Pi)$ is cuspidal and symplectic,
\item $\Pi_{v_0}\simeq \pi$,
\item $\Pi_{w_0}$ is cuspidal and $E'$-distinguished for some quadratic extension $E'/\mathfrak{F}_{w_0}$, 
\end{enumerate} 
\item a quadratic extension $\mathfrak{E}/\mathfrak{F}$ such that:
\begin{enumerate}
\item $\mathfrak{E}_{v_0}/\mathfrak{F}_{v_0}\simeq A/F$ and $\mathfrak{E}_{w_0}/\mathfrak{F}_{w_0} \simeq E'/\mathfrak{F}_{w_0}$,
\item $\mathfrak{E}_v\simeq \mathfrak{F}_v\times \mathfrak{F}_v$ for every place $v\notin \{v_0,w_0\}$ such that $\Pi_v$ is not isomorphic to a principal series induced from a Borel subgroup of $\GL_m(\mathfrak{D}_v)\simeq \GL_n(\mathfrak{F}_v)$,
\item $\mathfrak{E}/\mathfrak{F}$ embeds into $\CM_m(\mathfrak{D})/\mathfrak{F}$. 
\end{enumerate}
\end{enumerate}
\end{lemma}
\begin{proof}
(1) The existence of $\mathfrak{F}$ is well-known. 

(2) The existence of the division algebra satisfying (a), (b) and (c) follows from the Brauer-Hasse-Noether theorem.  

(3) We fix a quadratic extension $E'/\mathfrak{F}_{w_0}$, and an extra auxiliary finite place $v_1\neq w_0,v_0$. As discussed before the statement, there exists a cuspidal $E'$-distinguished representation $\pi_{w_0}$ of $\GL_m(\mathfrak{D}_{w_0})$, and an $\mathfrak{F}_{v_1}\times \mathfrak{F}_{v_1}$-distinguished cuspidal representation $\pi_{v_1}$ of $\GL_m(\mathfrak{D}_{v_1})\simeq \GL_n(\mathfrak{F}_{v_1})$. Now both $\JL(\pi)$ and $\JL(\pi_{w_0})$ are symplectic square-integrable representations, and $\JL(\pi_{v_1})=\pi_{v_1}$ is a cuspidal symplectic representation. We can thus apply \cite[Theorem 1.6(ii)]{TW} with their assumption (A2) holding at the place $v_1$ here (and $v_0$ there), to claim that there exists a self-dual cuspidal automorphic representation $\Pi'$ of $\GL_n(\A_\mathfrak{F})$ such that $\Pi'_{v_0}=\JL(\pi)$ and $\Pi'_{w_0}=\JL(\pi_{w_0})$ (and $\Pi'_{v_1}=\pi_{v_1}=\JL(\pi_{v_1})$). By \cite{JSEuler}, this implies that either the partial symmetric or exterior-square L-function of $\Pi'$ has a pole at $s=1$. It cannot be the symmetric-square L-function. Indeed otherwise by \cite[Theorem A]{KY}, the representation $\Pi'$ would be globally distinguished (in the terminology of [ibid.]), hence locally distinguished, which by \cite[Theorem A]{Y} would imply that $\phi_{\pi}$ is orthogonal. Since $\phi_{\pi}$ is irreducible, this contradicts our assumption that it is symplectic. Now by \cite{JS90} we deduce that $\Pi'$ is symplectic. Finally by \cite[Section 18]{BR}, there exists a cuspidal automorphic representation $\Pi$ of $\GL_m(\mathfrak{D}_{\A_{\mathfrak{F}}})$ such that (a) $\Pi'=\JL(\Pi)$. By the same reference (b) $\Pi_{v_0}=\pi$ and (c) $\Pi_{w_0}=\pi_{w_0}$.

(4) The existence of $\mathfrak{E}$ satisfying (a) and (b) follows from Krasner's lemma and weak approximation, as in \cite[Section 9.6]{AKMSS}. Then (c) automatically follows from \cite[Theorem 1.1]{SYY}; see \cite[Proof of Theorem 10.6]{MOY} for more details.
\end{proof}

The upshot of the above lemma is the following.

\begin{fact}\label{fct}
By construction, the cuspidal automorphic representation $\Pi$ in Lemma \ref{lm:globlz} is such that $\Pi_v$ is $\mathfrak{E}_v$-distinguished 
for all $v\neq v_0$. 
\end{fact}
\begin{proof}
It suffices to consider a place $v\neq v_0,w_0$. Note that $\mathfrak{D}_v$ is split by our choice of $\mathfrak{D}$, and that $\Pi_v=\JL(\Pi_v)=\JL(\Pi)_v$ is symplectic by the discussion at the beginning of Section \ref{sec:conv}. Hence if $\mathfrak{E}_v\simeq \mathfrak{F}_v\times \mathfrak{F}_v$, the result follows from the equality $\gadist(\GL_n(\mathfrak{F}_v))=\symp(\GL_n(\mathfrak{F}_v))$, with $A=\mathfrak{F}_v\times \mathfrak{F}_v$, recalled at the beginning of Section \ref{sec:gendist}. If not, i.e. if $\mathfrak{E}_v/\mathfrak{F}_v$ is quadratic, then by assumption $\Pi_v$ is a principal series induced from a Borel subgroup. The result then follows from \cite[Theorem 10.4]{MOY}.
\end{proof}

We are now in a position to compute $c_A(\pi,\psi)$ when $\pi$ is cuspidal symplectic. 

\begin{theorem}\label{thm:key}
Let $\pi\in \sympc(G)$. Then 
\[c_A(\pi,\psi)=(-1)^{m}\eta_{A/F}(-1)^{n/2}.\]
\end{theorem}
\begin{proof}
We take $\mathfrak{E}/\mathfrak{F}$, $\mathfrak{D}$, and $\Pi$, as in Lemma \ref{lm:globlz}, and $\Psi$ as in Lemma \ref{lm:product of c is one}. Note that since $\pi$ has a trivial central character, from its definition, the constant $c_A(\pi,\psi)$ does not depend on $\psi$. It is thus permissible to assume that $\Psi_{v_0}=\psi$. We denote by $S$ the set of ramified places as in Lemma \ref{lm:product of c is one}; it is a finite set containing $\{v_0,w_0\}$. By Lemma \ref{lm:product of c is one} and with its notations, we know that \[\prod_v c_{\mathfrak{E}_v}(\Pi_v,\Psi_v)=1.\] For $v\neq v_0,w_0$, the algebra $\mathfrak D_v$ is split, so the
integer playing the role of $m$ in Corollary~\ref{cor:key} is $n$;
the corresponding sign is therefore $(-1)^n=1$. At $w_0$, this
integer is $m$, and the corresponding sign is $(-1)^m$. Fact
\ref{fct} and Corollary~\ref{cor:key} thus imply that \[c_A(\pi,\psi)=(-1)^m\prod_{v\in S-\{v_0\}}\eta_{\mathfrak{E}_v/\mathfrak{F}_v}(-1)^{n/2}\]
\[=(-1)^m\prod_{v\neq v_0}\eta_{\mathfrak{E}_v/\mathfrak{F}_v}(-1)^{n/2}=(-1)^m\eta_{A/F}(-1)^{n/2},\] since $\eta_{\mathfrak{E}/\mathfrak{F}}(-1)=1$ by automorphy. 
\end{proof}

Theorem \ref{cusp thm}, from which Theorem \ref{conj:PTB} follows as explained in the introduction, is an immediate corollary of the above results.

\begin{corollary}\label{cor:main}
Let $\pi\in\cusp(G)$ and let $\phi=\phi_\pi$ be the corresponding $L$-parameter.
The following two conditions are equivalent.
\begin{itemize}
\item[(i)] $\pi$ is $H_A$-distinguished.
\item[(ii)] $\phi$ takes values in $\Sp_n(\C)$ and satisfies $\e_A(\phi)=(-1)^m$.
\end{itemize}
\end{corollary}
\begin{proof}
We only need to prove that (ii) implies (i). We assume that $\phi$ takes values in $\Sp_n(\C)$ and satisfies $\e_A(\phi)=(-1)^m$. Then Theorem \ref{thm:key} asserts that $\beta_{\pi,A,\psi}(0)=1$, and Theorem \ref{thm:key1} further implies that $J_{\pi,A}(s)$ has a pole at $s=0$; we conclude thanks to Proposition \ref{prop:key} that $\pi$ is $A$-distinguished. 
\end{proof}

\subsection{The generic statement, converse  implication}\label{sec:generic epsilon conv}

\begin{proof}[Proof of Theorem \ref{conj:PTB}, converse implication]
When $A\simeq F\times F$, it again follows from \cite[Theorem 4.6]{MOY}, whereas when $A/F$ is quadratic it follows from Corollary \ref{cor:main} and \cite{SX24}. 
\end{proof}

\section{The quaternionic symplectic model}\label{sec:quaternionic symplectic}

In this section $F$ is allowed to be of positive characteristic different from $2$, and $D$ is the quaternion division algebra
over $F$ together with its natural $F$-conjugation $x\to \overline{x}$. For $r\geq 1$ we set
\[
\begin{gathered}
G_r=\GL_r(D),\quad
J_r=(\delta_{i,r+1-j})_{1\leq i,j\leq r}, \quad 
H_r=\{g\in G_r:\,gJ_r\,{}^t\!\overline g=J_r\}.
\end{gathered}
\]

The proof of the following result is essentially given in the introduction.

\begin{theorem}\label{thm:quaternionic converse}
Let $\pi$ be an irreducible cuspidal representation of $G_m$.  If
$\JL(\pi)$ is not cuspidal, then
\[
 \Hom_{H_m}(\pi,\C)\neq 0.
\]
\end{theorem}
\begin{proof}
Let $I(s)=\pi[s]\times \pi[-s]$. By \cite[Section 2]{Tadic90}, since $\JL(\pi)$ is not cuspidal, the reducibility number of the cuspidal line of $\pi$ is equal to $l_\pi=2$. The representation $I(-1/2)$ is thus irreducible, hence there exists $M:I(-1/2)\simeq I(1/2).$  By \cite[Theorem 4.2]{ShV}, the representation $I(1/2)$ is distinguished since it affords a nonzero $H_{2m}$-invariant closed intertwining period $L: I(1/2)\to \C$. Hence $L\circ M\in \Hom_{H_{2m}}( I(-1/2),\C )-\{0\}$. But by the explicit geometric lemma for cuspidally induced representations in \cite[Proposition 4.2 and Lemma 4.1]{ShV}, which is valid in any characteristic different from $2$, the only possibilities of  $H_{2m}$-distinction for $I(-1/2)$ are either that $\pi[-1/2]\simeq \pi[3/2]$, which does not hold for central character reasons, or that $\pi$ is distinguished by $H_m$. The result follows.
\end{proof}

\newcommand{\etalchar}[1]{$^{#1}$}
\providecommand{\bysame}{\leavevmode\hbox to3em{\hrulefill}\thinspace}
\providecommand{\MR}{\relax\ifhmode\unskip\space\fi MR }
\providecommand{\MRhref}[2]{%
  \href{http://www.ams.org/mathscinet-getitem?mr=#1}{#2}
}
\providecommand{\href}[2]{#2}


\begin{thebibliography}{AKM{\etalchar{+}}21}

\bibitem[AC89]{AC}
James Arthur and Laurent Clozel, \emph{Simple algebras, base change, and the
  advanced theory of the trace formula}, Annals of Mathematics Studies, vol.
  120, Princeton University Press, Princeton, NJ, 1989. \MR{1007299}

\bibitem[AKM{\etalchar{+}}21]{AKMSS}
U.~K. Anandavardhanan, R.~Kurinczuk, N.~Matringe, V.~S\'echerre, and
  S.~Stevens, \emph{Galois self-dual cuspidal types and {A}sai local factors},
  J. Eur. Math. Soc. (JEMS) \textbf{23} (2021), no.~9, 3129--3191. \MR{4275472}

\bibitem[ALM{\etalchar{+}}24]{ALMSY}
U.~K. Anandavardhanan, Hengfei Lu, Nadir Matringe, Vincent S{\'e}cherre, and
  Chang Yang, \emph{The sign of linear periods}, Preprint, {arXiv}:2402.12106
  [math.{RT}], 2024.

\bibitem[BD92]{BrD}
Jean-Luc Brylinski and Patrick Delorme, \emph{Vecteurs distributions
  {$H$}-invariants pour les s\'eries principales g\'en\'eralis\'ees d'espaces
  sym\'etriques r\'eductifs et prolongement m\'eromorphe d'int\'egrales
  d'{E}isenstein}, Invent. Math. \textbf{109} (1992), no.~3, 619--664.
  \MR{1176208}

\bibitem[BD08]{BD}
Philippe Blanc and Patrick Delorme, \emph{Vecteurs distributions
  {$H$}-invariants de repr\'esentations induites, pour un espace sym\'etrique
  r\'eductif {$p$}-adique {$G/H$}}, Ann. Inst. Fourier (Grenoble) \textbf{58}
  (2008), no.~1, 213--261. \MR{2401221}

\bibitem[BP18]{BPgal}
Rapha{\"e}l Beuzart-Plessis, \emph{On distinguished square-integrable
  representations for {Galois} pairs and a conjecture of {Prasad}}, Invent.
  Math. \textbf{214} (2018), no.~1, 437--521 (English).

\bibitem[BPW19]{BPW}
Rapha{\"e}l Beuzart-Plessis and Chen Wan, \emph{A local trace formula for the
  generalized {Shalika} model}, Duke Math. J. \textbf{168} (2019), no.~7,
  1303--1385 (English).

\bibitem[BR10]{BR}
A.~I. Badulescu and D.~Renard, \emph{Unitary dual of {${\rm GL}(n)$} at
  {A}rchimedean places and global {J}acquet-{L}anglands correspondence},
  Compos. Math. \textbf{146} (2010), no.~5, 1115--1164. \MR{2684298}

\bibitem[CD94]{CrD}
Jacques Carmona and Patrick Delorme, \emph{Base m\'eromorphe de vecteurs
  distributions {$H$}-invariants pour les s\'eries principales
  g\'en\'eralis\'ees d'espaces sym\'etriques r\'eductifs: équation
  fonctionnelle}, J. Funct. Anal. \textbf{122} (1994), no.~1, 152--221.
  \MR{1274587}

\bibitem[Cho19]{Ch}
Marion Chommaux, \emph{Distinction of the {S}teinberg representation and a
  conjecture of {P}rasad and {T}akloo-{B}ighash}, J. Number Theory \textbf{202}
  (2019), 200--219. \MR{3958071}

\bibitem[CM22]{CM}
Marion Chommaux and Nadir Matringe, \emph{The split case of the
  {P}rasad-{T}akloo-{B}ighash conjecture for cuspidal representations of level
  zero}, Ann. Inst. Fourier (Grenoble) \textbf{72} (2022), no.~1, 123--153.
  \MR{4448593}

\bibitem[CM26]{ChenMatringe}
Fei Chen and Nadir Matringe, \emph{The algebraic theory of non-archimedean
  intertwining periods}, https://arxiv.org/abs/2610.17783 (2026).

\bibitem[CST17]{CST}
J.~W. Cogdell, F.~Shahidi, and T.-L. Tsai, \emph{Local {L}anglands
  correspondence for {${\rm GL}_n$} and the exterior and symmetric square
  {$\varepsilon$}-factors}, Duke Math. J. \textbf{166} (2017), no.~11,
  2053--2132. \MR{3694565}

\bibitem[DKV84]{DKV}
Pierre Deligne, David Kazhdan, and Marie-France Vign\'eras,
  \emph{Repr\'esentations des alg\`ebres centrales simples $p$-adiques},
  Repr\'esentations des groupes r\'eductifs sur un corps local, vol. Travaux en
  Cours, Hermann, Paris, 1984, pp.~33--117.

\bibitem[FH94]{FH}
Yuval~Z. Flicker and Jeffrey~L. Hakim, \emph{Quaternionic distinguished
  representations}, Am. J. Math. \textbf{116} (1994), no.~3, 683--736
  (English).

\bibitem[FLO12]{FLO}
Brooke Feigon, Erez Lapid, and Omer Offen, \emph{On representations
  distinguished by unitary groups}, Publ. Math., Inst. Hautes {\'E}tud. Sci.
  \textbf{115} (2012), 185--323 (English).

\bibitem[FMW18]{FMW}
Brooke Feigon, Kimball Martin, and David Whitehouse, \emph{Periods and
  nonvanishing of central {$L$}-values for {${\rm GL}(2n)$}}, Israel J. Math.
  \textbf{225} (2018), no.~1, 223--266. \MR{3805647}

\bibitem[Guo96]{Guo96}
Jiandong Guo, \emph{On a generalization of a result of {W}aldspurger}, Canad.
  J. Math. \textbf{48} (1996), no.~1, 105--142. \MR{1382478}

\bibitem[Hen00]{Hen}
Guy Henniart, \emph{Une preuve simple des conjectures de {L}anglands pour
  {$\mathrm{GL}(n)$} sur un corps {$p$}-adique}, Invent. Math. \textbf{139}
  (2000), no.~2, 439--455.

\bibitem[Hen10]{Hext}
\bysame, \emph{Correspondance de {L}anglands et fonctions {$L$} des carr\'es
  ext\'erieur et sym\'etrique}, Int. Math. Res. Not. IMRN (2010), no.~4,
  633--673. \MR{2595008}

\bibitem[HT01]{HT}
Michael Harris and Richard Taylor, \emph{The geometry and cohomology of some
  simple {S}himura varieties}, Annals of Mathematics Studies, vol. 151,
  Princeton University Press, Princeton, NJ, 2001.

\bibitem[JPSS83]{JPSS}
Herv{\'e} Jacquet, I.~I. Piatetskii-Shapiro, and Joseph~A. Shalika,
  \emph{Rankin--Selberg convolutions}, Amer. J. Math. \textbf{105} (1983),
  no.~2, 367--464. \MR{701565}

\bibitem[JS81]{JSEuler}
H.~Jacquet and J.~A. Shalika, \emph{On {E}uler products and the classification
  of automorphic representations. {I}}, Amer. J. Math. \textbf{103} (1981),
  no.~3, 499--558. \MR{618323}

\bibitem[JS90]{JS90}
Herv\'e Jacquet and Joseph Shalika, \emph{Exterior square {$L$}-functions},
  Automorphic forms, {S}himura varieties, and {$L$}-functions, {V}ol.\ {II}
  ({A}nn {A}rbor, {MI}, 1988), Perspect. Math., vol.~11, Academic Press,
  Boston, MA, 1990, pp.~143--226. \MR{1044830}

\bibitem[KY17]{KY}
Eyal Kaplan and Shunsuke Yamana, \emph{Twisted symmetric square {$L$}-functions
  for {${\rm GL}_n$} and invariant trilinear forms}, Math. Z. \textbf{285}
  (2017), no.~3-4, 739--793. \MR{3623730}

\bibitem[LM16]{LM}
Erez Lapid and Alberto M\'inguez, \emph{On parabolic induction on inner forms
  of the general linear group over a non-archimedean local field}, Selecta
  Math. (N.S.) \textbf{22} (2016), no.~4, 2347--2400. \MR{3573961}

\bibitem[Mat14]{MatJNT}
Nadir Matringe, \emph{Linear and {Shalika} local periods for the mirabolic
  group, and some consequences}, J. Number Theory \textbf{138} (2014), 1--19
  (English).

\bibitem[Mat15]{MatCrelle}
\bysame, \emph{On the local {B}ump-{F}riedberg {$L$}-function}, J. Reine Angew.
  Math. \textbf{709} (2015), 119--170. \MR{3430877}

\bibitem[Mat17]{MatBLMS}
\bysame, \emph{Shalika periods and parabolic induction for {$GL(n)$} over a
  non-archimedean local field}, Bull. Lond. Math. Soc. \textbf{49} (2017),
  no.~3, 417--427. \MR{3723627}

\bibitem[Mat21]{MatJFA}
\bysame, \emph{Gamma factors of intertwining periods and distinction for inner
  forms of {$\mathrm{GL}(n)$}}, J. Funct. Anal. \textbf{281} (2021), no.~10,
  Paper No. 109223, 70. \MR{4308058}

\bibitem[MOY24]{MOYjfa}
Nadir Matringe, Omer Offen, and Chang Yang, \emph{On local intertwining
  periods}, J. Funct. Anal. \textbf{286} (2024), no.~4, Paper No. 110293, 28.
  \MR{4679384}

\bibitem[MOY25]{MOY}
N.~Matringe, O.~Offen, and C.~Yang, \emph{Intertwining periods, {L}-functions
  and local-global principles for distinction of automorphic representations},
  https://arxiv.org/abs/2509.00441, 2025.

\bibitem[MSSS26]{MSSS}
Nadir Matringe, Vincent S{\'e}cherre, Shaun Stevens, and Miyu Suzuki,
  \emph{Discrete series representations of quaternionic {${\rm GL}_n(D)$} with
  symplectic periods}, To appear in Crelle, {arXiv}:2503.08955 [math.{RT}],
  2026.

\bibitem[MW89]{MW89}
Colette M{\oe}glin and Jean-Loup Waldspurger, \emph{Le spectre r{\'e}siduel de
  {$\mathrm{GL}(n)$}}, Ann. Sci. {\'E}cole Norm. Sup. (4) \textbf{22} (1989),
  no.~4, 605--674.

\bibitem[Ols74]{Olsh}
G.~I. Olshanski, \emph{Intertwining operators and complementary series in the
  class of representations of the full matrix group over a locally compact
  division algebra that are induced by parabolic subgroups}, Mat. Sb. (N.S.)
  \textbf{93(135)} (1974), 218--253, 326. \MR{499010}

\bibitem[Pra20]{Prfinite}
Dipendra Prasad, \emph{Multiplicities under basechange: finite field case}, J.
  Algebra \textbf{556} (2020), 1101--1114. \MR{4089564}

\bibitem[PTB11]{PTB}
Dipendra Prasad and Ramin Takloo-Bighash, \emph{Bessel models for
  {$\mathrm{GSp}(4)$}}, J. Reine Angew. Math. \textbf{655} (2011), 189--243.

\bibitem[Sai93]{Sai}
Hiroshi Saito, \emph{On {T}unnell's formula for characters of {${\rm GL}(2)$}},
  Compositio Math. \textbf{85} (1993), no.~1, 99--108. \MR{1199206}

\bibitem[S{\'e}c24]{Sec24}
Vincent S{\'e}cherre, \emph{Repr{\'e}sentations cuspidales de
  {$\mathrm{GL}_r(D)$} distingu{\'e}es par une involution int{\'e}rieure}, Ann.
  Sci. {\'E}c. Norm. Sup{\'e}r. (4) \textbf{57} (2024), no.~4, 961--1038.

\bibitem[Sha81]{Sha81}
Freydoon Shahidi, \emph{On certain {$L$}-functions}, Amer. J. Math.
  \textbf{103} (1981), no.~2, 297--355. \MR{610479}

\bibitem[Sha85]{Shreal}
\bysame, \emph{Local coefficients as {A}rtin factors for real groups}, Duke
  Math. J. \textbf{52} (1985), no.~4, 973--1007. \MR{816396}

\bibitem[Sha90]{Shproof}
\bysame, \emph{A proof of {L}anglands' conjecture on {P}lancherel measures;
  complementary series for {$p$}-adic groups}, Ann. of Math. (2) \textbf{132}
  (1990), no.~2, 273--330. \MR{1070599}

\bibitem[Sha26]{Sha26}
H.~Sharma, \emph{Symplectic model for {Z}elevinsky modules of
	{${\rm GL}(n,{\rm D})$}}, Algebr. Represent. Theory (2026),
doi:10.1007/s10468-026-10421-5.

\bibitem[ST23]{ST23}
M.~Suzuki and H.~Tamori, \emph{Epsilon dichotomy for linear models: the
  {A}rchimedean case}, Int. Math. Res. Not. (2023), no.~20, 17853--17891.

\bibitem[Suz21]{Sjnt}
Miyu Suzuki, \emph{Classification of standard modules with linear periods}, J.
  Number Theory \textbf{218} (2021), 302--310. \MR{4157701}

\bibitem[Suz25]{Strace}
\bysame, \emph{A reformulation of the conjecture of {P}rasad and
  {T}akloo-{B}ighash}, J. Algebra \textbf{680} (2025), 174--204. \MR{4913654}

\bibitem[SV17]{SV}
Yiannis Sakellaridis and Akshay Venkatesh, \emph{Periods and harmonic analysis
  on spherical varieties}, Ast{\'e}risque, vol. 396, Paris: Soci{\'e}t{\'e}
  Math{\'e}matique de France (SMF), 2017 (English).
  
  \bibitem[SV25]{ShV25}
  H.~Sharma and M.~K. Verma, \emph{Symplectic period for a representation of
  	{${\rm GL}_n({\rm D})$}}, J. Pure Appl. Algebra \textbf{229} (2025),
  Paper No.~107976.

\bibitem[SV26]{ShV}
Hariom Sharma and Mahendra~Kumar Verma, \emph{Symplectic model for ladder and
  unitary representations}, J. Pure Appl. Algebra \textbf{230} (2026), Paper
  No. 108334.

\bibitem[SX24]{SX24}
M.~Suzuki and H.~Xue, \emph{Linear intertwining periods and epsilon dichotomy
  for linear models}, Math. Ann. \textbf{388} (2024), 3589--3626.

\bibitem[SYY14]{SYY}
Sheng-Chi Shih, Tse-Chung Yang, and Chia-Fu Yu, \emph{Embeddings of fields into
  simple algebras over global fields}, Asian J. Math. \textbf{18} (2014),
  no.~2, 365--386. \MR{3217641}

\bibitem[Tad90]{Tadic90}
Marko Tadi{\'c}, \emph{Induced representations of {$\mathrm{GL}(n,A)$} for
  {$p$}-adic division algebras {$A$}}, J. Reine Angew. Math. \textbf{405}
  (1990), 48--77. \MR{1040995}

\bibitem[Tat79]{Tate}
John Tate, \emph{Number theoretic background}, Automorphic forms,
  representations, and {$L$}-functions, Part 2 (Armand Borel and W.~Casselman,
  eds.), Proceedings of Symposia in Pure Mathematics, vol.~33, American
  Mathematical Society, Providence, RI, 1979, pp.~3--26.

\bibitem[Tun83]{Tu}
Jerrold~B. Tunnell, \emph{Local {$\epsilon $}-factors and characters of {${\rm
  GL}(2)$}}, Amer. J. Math. \textbf{105} (1983), no.~6, 1277--1307. \MR{721997}

\bibitem[TW24]{TW}
Y.~Takanashi and S.~Wakatsuki, \emph{Asymptotic behavior for twisted traces of
  self-dual and conjugate self-dual representations of {$\mathrm{GL}_n$}},
  arXiv:2402.11945, 2024.

\bibitem[Ver18]{Verma}
Mahendra~Kumar Verma, \emph{On symplectic periods for inner forms of
  {$\mathrm{GL}_n$}}, Math. Res. Lett. \textbf{25} (2018), no.~1, 309--334.

\bibitem[Wal92]{RRG2}
Nolan~R. Wallach, \emph{Real reductive groups. {II}}, Pure and Applied
  Mathematics, vol. 132-II, Academic Press, Inc., Boston, MA, 1992.
  \MR{1170566}

\bibitem[Wan22]{Wan}
Chen Wan, \emph{On a multiplicity formula for spherical varieties}, J. Eur.
  Math. Soc. (JEMS) \textbf{24} (2022), no.~10, 3629--3678. \MR{4432908}

\bibitem[Xue21]{Xue21}
Hang Xue, \emph{Epsilon dichotomy for linear models}, Algebra Number Theory
  \textbf{15} (2021), no.~1, 173--215.

\bibitem[Yam17]{Y}
Shunsuke Yamana, \emph{Local symmetric square {$L$}-factors of representations
  of general linear groups}, Pacific J. Math. \textbf{286} (2017), no.~1,
  215--256. \MR{3582406}

\end{thebibliography}
\end{document}